\documentclass[11pt]{article}
\usepackage{amssymb, amsmath, mathrsfs, amsthm}
\usepackage{graphicx}
\usepackage{color}
\usepackage[top=1in, bottom=1in, left=1in, right=1in]{geometry}
\usepackage{float, caption, subcaption}
\usepackage{enumitem}
\usepackage{setspace}
\usepackage{fancyhdr}
\usepackage{verbatim}
\usepackage{authblk}
\graphicspath{ {./images/} }
\usepackage{subfiles} % Best loaded last in the preamble
\usepackage{tikz}
\usetikzlibrary{decorations.pathreplacing,calligraphy}
\usetikzlibrary{shapes.geometric}

\newtheorem{theorem}{Theorem}[section]
\newtheorem{lemma}[theorem]{Lemma}

\newtheorem{proposition}{Proposition}[section]

\theoremstyle{definition}
\newtheorem{definition}{Definition}[section]

\theoremstyle{remark}
\newtheorem{remark}{Remark}[section]

\newcommand{\dd}{\mathsf{d}}

\begin{document}

\title{Extremal values for Steiner distances and the Steiner $k$-Wiener index}

\author{Hua Wang\footnote{Georgia Southern University, Statesboro, GA 30460, USA {\tt hwang@georgiasouthern.edu}}, Andrew Zhang\footnote{(Corresponding author) Wayzata High School, Plymouth, MN 55446, USA \\ {\tt andrewzhang05719@gmail.com}}}
\date{ }
\maketitle
	
%\maketitle
\begin{abstract}

Various questions related to distances between vertices of simple, finite graphs are of interest to extremal graph theorists. The Steiner distance of a set of $k$ vertices is a natural generalization of the regular distance. We extend several theorems on the middle parts and extremal values of trees from their regular distance variants to their Steiner distance variants. More specifically, we show that for a tree $T$, the Steiner $k$-distance, Steiner $k$-leaf-distance, and Steiner $k$-internal-distance are all concave along a path. We also calculate distances between the Steiner $k$-median, Steiner $k$-internal-median, and Steiner $k$-leaf-median. Letting the Steiner $k$-distance of a vertex $v \in V(T)$ be $\dd_k^T(v)$, we find bounds based on the order of $T$ for the ratios $\frac{\dd^T_{k}(u)}{\dd^T_{k}(v)}$, $\frac{\dd^T_{k}(w)}{\dd^T_{k}(z)}$, and $\frac{\dd^T_{k}(u)}{\dd^T_{k}(y)}$ where $u$ and $v$ are leaves, $w$ and $z$ are internal vertices, and $y$ is a Steiner $k$-centroid. Also, denoting the Steiner $k$-Wiener index as $\mathsf{SW}_k(T)$, we find upper and lower bounds for $\frac{\mathsf{SW}_k(T)}{\dd^G_{k}(v)}$. The extremal graphs that produce these bounds are also presented.

%These work generalizes previous results on regular distances and their ratios, established in [Centroid, leaf-centroid, and internal-centroid. Graphs and Combinatorics, 31(3): 783--
%793, May 2015.] and [Extremal values for ratios of distances in
%trees. Discrete Applied Mathematics, 80(1): 37--56, 1997.]

%Since various structures can be modeled using graphs, our results can help guide the creation of networks that maximize the efficacy of central facilities.     
\end{abstract}








% \end{thebibliography}

%\onehalfspacing

%\title{Extremal values for Steiner distances and the Steiner $k$-Wiener index}
%\author[]{}
%\author[2]{Hua Wang}

%\affil[]{Wayzata High School, United States}
%\affil[2]{Georgia Southern University, United States }
%\date{October 2022}

\section{Introduction}
Distance-based topological indices are one of the best known graph invariants studied. For two vertices in a graph, the distance is the number of edges in the shortest path between them. The Steiner distance, defined in \cite{Chartrand1989}, generalizes this concept of distance as it is the size of the minimal graph that connects a set of vertices, rather than just a pair. Many questions related to the graph structure or extremal problems can be generalized from regular distances to the Steiner distances.

\subsection{Middle parts}

These variants of distance can be used to define several indices that determine ``middle parts" of graphs, including the eccentricity, center, and centroid. For a given vertex, its eccentricity is its distance from the vertex furthest from it. The eccentricity can be used as a measure of how centralized a vertex is within a graph and has been investigated in \cite{dankemann04}. Intuitively, the center is defined as the subgraph formed by the vertices with minimal eccentricity. However, while the center is a useful indication of the centralized vertices in a graph, it only considers a vertex's distance to one other vertex. The centroid, introduced in \cite{Jordan1869}, is an alternative method for describing centralized vertices that considers a vertex's distance in relation to all other vertices in the graph. In \cite{Wang2015}, new variants of the centroid including the leaf-centroid and internal-centroid are introduced. We generalize those concepts through the Steiner distance, and prove that for any three vertices $x, y, z$ of a tree $T$ with $xy, yz \in E(T)$, we must have $ 2\dd_k(y) \leq \dd_k(x) + \dd_k(z)$, $2\mathsf{dL}_k(y) \leq \mathsf{dL}_k(x) + \mathsf{dL}_k(z)$, and $2\mathsf{dI}_k(y) \leq \mathsf{dI}_k(x) + \mathsf{dI}_k(z)$. Detailed definition of notations will follow in the next section.

The Steiner distance can be used to define the Steiner $k$-centroid vertices, generalizing those from regular centroids. The Steiner $k$-median, which is the subgraph induced by the Steiner $k$-centroids, was first introduced in \cite{berineke1996}. In this paper, we generalize results in \cite{Wang2015} by analyzing Steiner $k$-leaf-centroids and Steiner $k$-internal-centroids. Additionally, we study the distance between the Steiner $k$-median, Steiner $k$-leaf-median, and Steiner $k$-internal-median. For instance, between the Steiner $k$-leaf-median and Steiner $k$-internal-median we have
\begin{equation*}\dd(u, v) \leq \max\{0, n + 1- \max \{k,3\} - \max\{k+1, \dd(u, v) + 2\}\}.
\end{equation*}
Similarly, we also have
\begin{equation*}
\dd(u, v) \leq \frac{n-2\max \{k,3\}+1}{2}
\end{equation*}
between the leaf-median and median, and
\begin{equation*}
  \dd(u,v) \leq \frac{n - 2 \max \{ k+1, \dd(u, v) + 2 \} + 1}{2}
\end{equation*}
between the internal-median and median.

\subsection{The Wiener index and extremal ratios}

The most well studied distance-based topological index is probably the Wiener index, introduced by Wiener in \cite{winer47} as the sum of distances between all pairs of vertices. The Steiner $k$-Wiener index is a natural generalization, defined using the Steiner $k$-distance rather than regular distance. The Steiner $k$-Wiener index has been investigated in \cite{gutman16, zhang19, Rasila21}.
The Steiner $k$-distance can also be used to define other global topological indices on an entire graph. An example is the average Steiner $k$-distance, which has been studied in \cite{Dankelmann96,Dankelmann97}. 

Analyzing the extremal ratios between functions defined on graph invariants lead to a number of interesting questions, studied in \cite{barefoot97,szekely2013,szekely14}. 
In particular, the ratios between the regular distance and the regular Wiener index have been studied in \cite{barefoot97}. These ratios provide a measure of the equality between vertices in a graph. Generalizing these results to those defined through Steiner distances, we stduy explicit bounds and extremal structures corresponding to the extremal ratios of the Steiner distance functions. Our results include the extremal ratio between 
a pair of leaves, between a pair of internal vertices, and between one centroid and one leaf. It is interesting to see that the extremal structures are often the so called ``comets'', similar to those found in previous studies with other functions.

The extremal ratio between the global function and local function has also been studied in \cite{barefoot97}, for distance functions. The study of such ratios were extended to those regarding the number of subtrees in \cite{szekely2013,szekely14}. We generalize these results to the Steiner $k-$distance and the Steiner $k$-Wiener Index in this paper, showing that $\displaystyle{ \frac{n}{k} \leq \frac{\mathsf{SW}_k(T)}{\dd_k(v)} \leq \frac{n-1}{k-1}}$.

\subsection{Organization of content}

We will primarily operate on trees. In Section \ref{Sec: prelims}, we formally define the notation used throughout this paper. In Section \ref{Sec:concavity}, we prove that along a path, the Steiner $k$-distance, Steiner $k$-leaf-distance, and Steiner $k$-internal-distance of more middle vertices is on average less than the Steiner $k$-distance of the vertices closer to the end of the path. Then, in Section \ref{Sec:centroids}, we find distances between the Steiner $k$-medians, Steiner $k$-leaf-medians, and Steiner $k$-internal-medians. Finally, in Section \ref{Sec:ratio}, we bound the ratio of the Steiner $k$-distances of a pair of leaves, a pair of internal vertices, and a pair of a leaf and a Steiner $k$-centroid. We also bound the ratio between the Steiner $k$-Wiener index and the Steiner $k$-distance of a Steiner $k$-centroid.

\section{Preliminaries}\label{Sec: prelims}

In this paper, we consider only simple, finite graphs. Although all definitions below apply to general graphs, in this paper we will only consider them in the context of trees. For a graph $G$, let the set of vertices of $G$ be $V(G)$ and the set of edges of $G$ be $E(G)$. Also, denote the set of leaves of $G$ as $L(G)$ and the set of internal vertices of $G$ as $I(G)$. For a subgraph $C$ of $G$, we define $L'_G(C)$ as the set of vertices in $C$ that are leaves in $G$ and $I'_G(C)$ as the set of vertices in $C$ that are internal vertices in $G$. When the graph $G$ is given by context, we simply write $L'(C)$ and $I'(C)$.

For a pair of vertices $u,v$ we denote a path between them as $P(u,v)$. The distance between a pair of vertices $u$ and $v$, denoted by $\dd(u,v)$, is the length of the shortest path between them. The Steiner distance of a set of vertices $S$ in $G$, denoted by $\dd_G(S)$, is the size of the smallest connected subgraph in $G$ containing $S$. This subgraph will always be a tree and is referred to as the Steiner tree of $S$. When the graph is given by context, we refer to $\dd_G(S)$ as $\mathsf{d}(S)$. 

Throughout this paper, unless otherwise noted, we let $n$ be the order of the graph under consideration and $k$ be an integer such that $2 \leq k \leq n$. 
The Steiner $k$-Wiener index is a natural generalization of the Wiener index.
\begin{definition}
For a graph $G$, 
the Steiner $k$-Wiener index is 
\begin{equation*}
\mathsf{SW}_k(G) =  \sum_{S\subset V(G),\vert S\vert =k}
\dd(S).     
\end{equation*}
\end{definition}

While the Stiener $k$-Wiener index can be considered as a global function on graphs, the corresponding local function is the Steiner $k$-distance of vertices.
\begin{definition}
For a graph $G$, the Steiner $k$-distance of a vertex $v \in V(G)$, denoted by $\dd_k^{G}(v)$, is the sum of the Steiner distances of all sets of $k$ vertices in $G$ containing $v$, 
\begin{equation*}
\displaystyle{\dd_k^G(v) = \sum_{S \subset V(G), \text{ } |S| = k-1, \text{ } v \notin S} \dd_G(\{v\} \cup S)}.
\end{equation*}
When the graph is given by context, we simply write $\dd_k(v)$.
\end{definition}

This definition easily generalizes to the ``leaf'' and ``internal'' versions of the same concept.

\begin{definition}\label{Def: Steiner k-leaf-distance}
For a graph $G$, the Steiner $k$-leaf-distance of a vertex $v \in V(G)$, denoted by $\mathsf{dL}_k^{G}(v)$, is 
\begin{equation*}
{\mathsf{dL}_k^{G}(v) = \sum_{S \subset L(G), \text{ } |S| = k-1} \dd_G(\{v\} \cup S)}.
\end{equation*}
When the graph is given by context, we simply write $\mathsf{dL}_k(v)$.
\end{definition}
\begin{definition}\label{Def: Steiner k-internal-distance}
For a graph $G$, the Steiner $k$-internal-distance of a vertex $v \in V(G)$, denoted by $\mathsf{dI}_k^G(v)$, is 
 \begin{equation*}
  \mathsf{dI}_k^G(v) = \sum_{S \subset I(G), \text{ } |S| = k-1} \dd_G(\{v\} \cup S). 
  \end{equation*}
When the graph is given by context, we simply write $\mathsf{dI}_k(v)$. 
\end{definition}
\begin{remark}
If a vertex $v$ is a leaf, for a set $S$ as in Definition \ref{Def: Steiner k-leaf-distance}, $|\{v\} \cup S|$ could be $k-1$. If a vertex $v$ is an internal vertex, for a set $S$ as in Definition \ref{Def: Steiner k-internal-distance}, $|\{v\} \cup S|$ could be $k-1$.
\end{remark}

As mentioned in the previous section, the middle parts of a graph are defined based on these local functions.
\begin{definition}
For a graph $G$, we have the following.
\begin{itemize}
    \item A Steiner $k$-centroid (vertex) is a vertex with minimal Steiner $k$-distance, and the Steiner $k$-median $M_k(G)$ is the subgraph induced by the Steiner $k$-centroids.
    \item A Steiner $k$-leaf-centroid (vertex) is a vertex with minimal Steiner $k$-leaf-distance, and the Steiner $k$-leaf-median $M^\ell_k(G)$ is the subgraph induced by the Steiner $k$-leaf-centroids.
    \item A Steiner $k$-internal-centroid (vertex) is a vertex with minimal Steiner $k$-internal-distance, and the Steiner $k$-internal-median $M^i_k(G)$ is the subgraph induced by the Steiner $k$-internal-centroids.
\end{itemize}
\end{definition}
 
We also introduce some technical notations to help us analyze the Steiner $k$-distances of specific vertices.
\begin{definition}
Let $\mathsf{d}_k(a_1,a_2, \dots, a_i; b_1,b_2, \dots, b_j)$ be the sum of the Steiner distances of the sets of $k$ vertices containing the vertices $a_1,a_2, \dots, a_i$ and not containing the vertices $b_1,b_2, \dots, b_j$. 
\end{definition}

\begin{definition}
Let $\mathsf{dL}_k(a_1,a_2, \dots, a_i; b_1,b_2, \dots, b_j)$ be the sum of the Steiner distances of the sets of $k$ vertices containing the vertices $a_1,a_2, \dots, a_i$, not containing the vertices $b_1,b_2, \dots, b_j$, and not containing any internal vertices besides possibly $a_1,a_2, \dots, a_i$. 
\end{definition}

\begin{definition}
Let $\mathsf{dI}_k(a_1,a_2, \dots, a_i; b_1,b_2, \dots, b_j)$ be the sum of the Steiner distances of the sets of $k$ vertices containing the vertices $a_1,a_2, \dots, a_i$, not containing the vertices $b_1,b_2, \dots, b_j$, and not containing any leaves besides possibly $a_1,a_2, \dots, a_i$. 
\end{definition}

To describe some extremal graphs in this paper, we define the $r$-comet as follows.
\begin{definition}\label{Def: r-comet and vertices}
A $r$-comet on $n$ vertices is a tree formed by attaching $n-r$ leaves onto an end vertex of a path of length $r-1$. On a $r$-comet, let $a$ denote one of the $n-r$ leaves, $b$ denote the leaf on the path of length $r-1$, $c$ denote the vertex adjacent to $b$, and $d$ denote the vertex on the path of length $r-1$ adjacent to $n-r$ leaves as shown in Figure \ref{Fig:comet}.
\begin{figure}[htb]
\begin{center}
%\begin{tikzpicture}[thick, scale=0.8]

\begin{tikzpicture}[scale=0.7,
    brace/.style={thick,decorate,
        decoration={calligraphic brace, amplitude=7pt,raise=0.5ex}}]

\draw [decorate,decoration={brace,amplitude=10pt},xshift=0pt,yshift=8pt]
(2.5,0) -- (8.5,0)node [black,midway,yshift=16pt] {
$r \ vertices$};

\draw (1,-1) -- (2.5,0) 
 node [draw=black, circle, fill=black, inner sep=2pt, pos=0]  {};
\draw (1,0) -- (2.5,0) 
node [draw=black, circle, fill=black, inner sep=2pt, pos=0]  {};
\draw (1,1) -- (2.5,0)
 node [draw=black, circle, fill=black, inner sep=2pt, pos=0]  {};
\draw  (2.5, 0) -- (4,0)
 node [draw=black, circle, fill=black, inner sep=2pt, pos=1]  {}
  node [draw=black, circle, fill=black, inner sep=2pt, pos=0]  {};

\draw [dashed](4,0) -- (5.5,0)
 node [draw=black, circle, fill=black, inner sep=2pt, pos=1]  {};

\draw [dashed](5.5,0) -- (7,0)
 node [draw=black, circle, fill=black, inner sep=2pt, pos=1]  {};
 \draw (7,0) -- (8.5,0)
 node [draw=black, circle, fill=black, inner sep=2pt, pos=1]  {};

  \node [left] at (1,1) {$a$};
   \node [below] at (2.5,0) {$d$};
 \node [below] at (7,-0.1) {$c$};
\node [below] at (8.5,0) {$b$};

\end{tikzpicture}
\end{center}
 \caption{A $r$-comet with vertices $a, b$, $c$, and $d$.}
 \label{Fig:comet}
\end{figure}
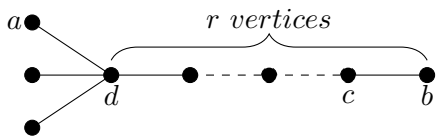
\end{definition}

 Throughout this paper, we use $G-v$ to denote the graph resulting from removing the vertex $v$ (and its adjacent edges) from $G$. Additionally, for $uv \in E(G)$, we define $G-uv$ to be the graph resulting from removing the edge $uv$ from $G$. For a vertex $x \in V(G)$, we generally use $C_x$ to denote the connected component containing $x$ after edge or vertex removals have disconnected $G$.

%In \cite{berineke1996} the Steiner $k$-distance where $k\geq 2$ of a vertex $v\in V(G)$ in a connected graph $G$ on $n\geq k$ vertices, denoted by $\dd_k(v)$ is defined by
%$\dd_k(v) = \displaystyle \sum_{S\subset V(G), \|S\| = k, v\in S } \dd_G(S)$.

\section{Concavity of Steiner Distances} \label{Sec:concavity}

A tree $T$ on $n$ vertices with $x,y,z \in V(T)$ such that $xy, yz \in E(T)$ can be represented as in Figure \ref{Figure: xyz Tree}. Let $C_x$, $C_y$, and $C_z$ be the connected components containing $x$, $y$, and $z$, respectively, in $T-xy-yz$. Furthermore, let $F_x = C_x - x$, 
$F_y = C_y - y$, and $F_z = C_z - z$. Finally, define $S_{a_1, \cdots, a_i; b_1, \cdots, b_j}^k(T)$ for $2 \leq k \leq n$ and $a_1, \cdots, a_i, b_1, \cdots, b_j \in \{ x,y,z \}$ to be the set containing all sets of $k$ vertices in $T$ that do not contain $x,y,z$ or any vertices in $ C_{b_1}, \cdots, C_{b_j}$ but do contain at least one vertex in each of $C_{a_1}, \cdots, C_{a_i}$. When the tree is given by context, we refer to $S_{a_1, \cdots, a_i; b_1, \cdots, b_j}^k(T)$ as 
$S_{a_1, \cdots, a_i; b_1, \cdots, b_j}^k$. We now prove that the Steiner $k$-distance, the Steiner $k$-leaf-distance, and the Steiner $k$-internal-distance are concave along a path.
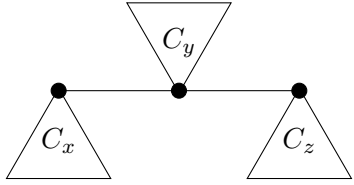
\begin{figure}[htp] 
    \centering
    \begin{center}
\begin{tikzpicture}[scale = 0.8]
\node [above] at (0,-1.2) {$C_x$};
\node [above] at (4,-1.2) {$C_z$};
\node [below] at (2,1.2) {$C_y$};
\draw(0,0) -- (2, 0) -- (4,0) 
 node [draw=black, circle, fill=black, inner sep=2pt, pos=0]  {}
  node [draw=black, circle, fill=black, inner sep=2pt, pos=1]  {}
   node [draw=black, circle, fill=black, inner sep=2pt, pos=-1]  {};
\node[isosceles triangle,isosceles triangle apex angle=60,
    draw,
    rotate=90,
    minimum size =1.2cm] (T1)at (0,-1){};
    \node[isosceles triangle,isosceles triangle apex angle=60,
    draw,
    rotate=270,
    minimum size =1.2cm] (T1)at (2,1){};
\node[isosceles triangle,isosceles triangle apex angle=60,
    draw,
    rotate=90,
    minimum size =1.2cm] (T1)at (4, -1){};

\end{tikzpicture}
\end{center}
    \caption{A tree with vertices $x,y,z$}
    \label{Figure: xyz Tree}
\end{figure}

\begin{proposition}
\label{Theorem: Steiner Distance Concavity}
    For a tree $T$ with vertices $x, y, z$ such that $xy, yz \in E(T)$ we have \begin{equation*}
    2\dd_k(y) \leq \dd_k(x) + \dd_k(z)
    \end{equation*}
    with equality if and only if $|V(F_y)| + |V(F_z)| < k-2$ and $|V(F_x)| + |V(F_y)| < k-2$
\end{proposition}

\begin{proof}
By definition, $\dd_k(x), \dd_k(y),$ and $\dd_k(z)$ each equal the sum of the Steiner distances of sets of $k$ vertices. We first consider the Steiner distance contributed toward these sums by sets containing exactly one of $x, y,$ or $z$.

Note that for each set $t \in S_{x;y,z}^{k-1}$, we have  $\dd(\{x \} \cup t)+ 2 = \dd(\{y \} \cup t) + 1 = \dd(\{z \} \cup t)$. Thus, $\dd(\{x \} \cup t) + \dd(\{z \} \cup t) = 2\dd(\{y \} \cup t)$. So,
\begin{equation*}
    \sum_{t \in S_{x;y,z}^{k-1} }\left( \dd(\{ x \} \cup t) + \dd (\{z \}  \cup t) \right) =  \sum_{t \in S_{x;y,z}^{k-1}}2\dd(\{y\}\cup t).
\end{equation*}
By symmetry,
\begin{equation*}
    \sum_{t \in S_{z;x,y}^{k-1} }\left( \dd(\{x\}  \cup t) + \dd (\{z\} \cup t)\right) =  \sum_{t \in S_{z;x,y}^{k-1}}2\dd(\{y\}\cup t).
\end{equation*}

For each set $t\in S_{y;x,z}^{k-1}$, we have $\dd(\{x \} \cup  t ) = \dd(\{y \} \cup t) + 1 = \dd(\{ z \} \cup t)$. So, $2\dd(\{ y \} \cup t) + 2 = \dd(\{x \} \cup t) + \dd(\{z \} \cup t)$. Since $|S_{y;x,z}^{k-1}| = \binom{|V(F_y)|}{k-1}$, 
\begin{equation*}
   \sum_{t\in S_{y;x,z}^{k-1} } \left(  \dd( \{ x \} \cup t) + \dd ( \{ z \} \cup t)\right) =  \sum_{t\in S_{y;x,z}^{k-1} } 2\dd(\{y \} \cup t) + 2 \binom{|V(F_y)|}{k-1}.
\end{equation*}
Thus, $\displaystyle{\sum_{t\in S_{y;x,z}^{k-1}} \left( \dd(\{x \} \cup t) + \dd (\{z \} \cup t)\right) =  \sum_{t\in S_{y;x,z}^{k-1} } 2\dd(\{y\} \cup t)}$ if and only if $|V(F_y)| < k - 1.$ 

For each set $t \in S_{x,y;z}^{k-1}$, we have $\dd(\{x \} \cup t) + 1 = \dd(\{y \} \cup t) + 1 = \dd(\{z \} \cup t)$. So, $2\dd(\{y \} \cup t) + 1 = \dd(\{x \} \cup t) + \dd(\{z \} \cup t)$. Since $|S_{x,y;z}^{k-1}|= \binom{|V(F_x)|+|V(F_y)|}{k-1} - \binom{|V(F_x)|}{k-1} - \binom{|V(F_y)|}{k-1}$, 
\begin{align*}
   &\sum_{t\in S_{x,y;z}^{k-1} } \left( \dd(\{x\} \cup t) + \dd (\{z\} \cup t) \right) = \\
     & \quad \quad \sum_{t\in S_{x,y;z}^{k-1} } 2\dd(\{y\} \cup t) + \binom{|V(F_x)|+|V(F_y)|}{k-1} - \binom{|V(F_x)|}{k-1} - \binom{|V(F_y)|}{k-1}.
\end{align*}
So, equality holds when $|V(F_x)| = 0, |V(F_y)| = 0,$ or $|V(F_x)|+|V(F_y)| < k-1$. By symmetry, 
\begin{align*}
  & \sum_{t\in S_{y,z;x}^{k-1} } \left(\dd(\{x\}  \cup t) + \dd (\{z\}  \cup t) \right)=  \\
  & \quad \quad \sum_{t\in S_{y,z;x}^{k-1} } 2\dd(\{y\} \cup t) + \binom{|V(F_y)|+|V(F_z)|}{k-1} - \binom{|V(F_y)|}{k-1} - \binom{|V(F_z)|}{k-1}.
\end{align*}
with equality when $|V(F_y)| = 0, |V(F_z)| = 0,$ or $|V(F_y)|+|V(F_z)| < k-1$. For each set $t \in S_{x,y,z;}^{k-1} \cup S_{x,z;y}^{k-1}$, we have that $\dd(\{x\} \cup t) = \dd(\{y\} \cup t) = \dd(\{z\} \cup t)$. Thus, 
\begin{equation*}
    \sum_{t \in S_{x,y,z;}^{k-1} \cup S_{x,z;y}^{k-1}}\left( \dd(\{x\}  \cup t) + \dd (\{z\} \cup t)\right) =  \sum_{t \in S_{x,y,z;}^{k-1} \cup S_{x,z;y}^{k-1}}2\dd(\{y\}  \cup t).
\end{equation*}

Now, consider the sets of $k$ vertices with more than one of $x,y,$ and $z$. Let $\dd'_k(x)$,  $\dd'_k(y)$, and $\dd'_k(z)$ be the sum of sets of $k$ vertices containing $x,y,$ and $z$, respectively, while also containing more than one of $x,y,$ and $z$. Note that $\dd'_k(x) = \dd_k(x,z;y) + \dd_k(x,y;z) + \dd_k(x,y,z;)$, $\dd'_k(y) = \dd_k(x,y;z) + \dd_k(y,z;x) + \dd_k(x,y,z;)$, and $\dd'_k(z) = \dd_k(x,z;y) + \dd_k(y,z;x) + \dd_k(x,y,z;)$. Thus, to prove that $2\dd'_k(y) \leq \dd'_k(x) + \dd_k'(z)$, it suffices to show $ \dd_k(x,y;z) + \dd_k(y,z;x) \leq 2\dd_k(x,z;y)$.

Note that $\dd_k(x,y;z) \leq \dd_k(x,z;y)$ and $\dd_k(y,z;x) \leq \dd_k(x,z;y)$ because $y$ must be in the Steiner tree of a set of $k$ vertices containing $x$ and $z$. We have $\dd_k(x,y;z) = \dd_k(x,z;y)$ if and only if $|V(F_x)| + |V(F_y)| < k-2$ as that forces all sets of $k-2$ vertices not including $x,y,$ or $z$ to have an element in $F_z$. Similarly, $\dd_k(y,z;x) = \dd_k(x,z;y)$ if and only if  $|V(F_y)| + |V(F_z)| < k-2$ as that forces all sets of $k-2$ vertices not including $x,y,$ or $z$ to have an element in $F_x$. Thus, combining all conditions for equality, we have $2\dd_k(y) \leq \dd_k(x) + \dd_k(z)$ with equality occurring when $|V(F_y)| + |V(F_z)| < k-2$ and $|V(F_x)| + |V(F_y)| < k-2$.
\end{proof}

\begin{remark}
The following Propositions~\ref{Theorem: Steiner Leaf Distance Concavity} and \ref{Theorem: Steiner Internal Distance Concavity} follow from similar logic. The conditions for equalities, while not difficult to obtain, are rather cumberson to state. We omit them here.
\end{remark}

\begin{proposition} \label{Theorem: Steiner Leaf Distance Concavity}
For a tree $T$ with vertices $x, y, z$ such that $xy, yz \in E(T)$ we have \begin{equation*}
2\mathsf{dL}_k(y) \leq \mathsf{dL}_k(x) + \mathsf{dL}_k(z).
\end{equation*}

\end{proposition}
\begin{proof}
Let $S_1$ be the set containing all sets of $k-1$ leaves not containing $x,y$ or $z$. Following similar argument as in Proposition \ref{Theorem: Steiner Distance Concavity}, we have
\begin{equation*}
    \sum_{t \in S_1}2\mathsf{dL}(y \cup S_1) \leq \sum_{t \in S_1}\left(\mathsf{dL}(x \cup S_1) + \mathsf{dL}(z \cup S_1)\right).
\end{equation*}

Now we account for sets of $k-1$ leaves containing at least one of $x,y$ or $z$. Let $\mathsf{dL}'_k(x)$ be the sum of Steiner distances of the sets generated by the union of $x$ and a set of $k-1$ leaves containing at least one of $x,y$ or $z$. Let $\mathsf{dL}'_k(y)$ and $\mathsf{dL}'_k(z)$ be defined similarly.

Recall that for a vertex $v \in L(T)$, the set $\{v\} \cup S$ where $S$ is a set of $k-1$ leaves will have $k-1$ elements if $v \in S$. As $y$ is adjacent to both $x$ and $z$ it is not a leaf. We now use casework on whether $x$ and $z$ are leaves. 
\begin{enumerate} [label={Case }{{\arabic*}:}]
    \item Neither $x$ or $z$ are leaves. In this scenario, we do not have any sets of $k-1$ leaves containing at least one of $x,y$ and $z$.
    \item Exactly one of $x$ or $z$ is a leaf. As the case where $x$ is a leaf but $z$ is not is symmetrical to the case where $z$ is a leaf but $x$ is not, we first consider when $x$ is a leaf but $z$ is not. In this case, we have $\mathsf{dL}'_k(x) = \mathsf{dL}_{k-1}(x;y,z)$, $\mathsf{dL}'_k(y) = \mathsf{dL}_k(x,y;z)$, and $\mathsf{dL}'_k(z) = \mathsf{dL}_k(x,z;y)$. We will now consider each possible set of $k-2$ vertices not containing $x,y,$ or $z$. Note that for each set $t \in S_{x;y,z}^{k-2}$, we have $2\dd(\{x,y\} \cup t) = \dd(\{x,z\} \cup t) + \dd(\{x\} \cup t)$ as $\dd(\{x\} \cup t) + 2 = \dd(\{x,y\} \cup t)+1 = \dd(\{x,z\} \cup t)$. Similarly, for each set $t \in S_{z;x,y}^{k-2}$, we have $2\dd(\{x,y\} \cup t) = \dd(\{x,z\} \cup t) + \dd(\{x\} \cup t)$ as $\dd(\{x\} \cup t) = \dd(\{x,y\} \cup t) = \dd(\{x,z\} \cup t)$.

    However, for each set $t \in S_{y;x,z}^{k-2}$, we have $\dd(\{x\} \cup t) + 1 = \dd(\{x,y\} \cup t) + 1 = \dd(\{x,z\} \cup t)$. So, $2\dd(\{x,y\} \cup t) + 1 = \dd(\{x\} \cup t) + \dd(\{x,z\} \cup t)$. Similarly, for each set $t \in S_{x,y;z}^{k-2}$, we have $\dd(\{x\} \cup t) + 1 = \dd(\{x,y\} \cup t) + 1 = \dd(\{x,z\} \cup t)$. For each set $t$ in $S_{x,z;y}^{k-2}, S_{y,z;x}^{k-2}$, and $S_{x,y,z;}^{k-2}$, we have $\dd(\{x\} \cup t) = \dd(\{x,y\} \cup t)= \dd(\{x,z\} \cup t)$. 
    
    Thus, as we have covered all sets of $k-2$ leaves not including $x,y$ or $z$, when $x$ is a leaf but $z$ is not, we have $2\mathsf{dL}'_k(y) \leq \mathsf{dL}'_k(x) + \mathsf{dL}'_k(z)$. Using similar logic, the case where $z$ is a leaf but $x$ is not has $2\mathsf{dL}'_k(y) \leq \mathsf{dL}'_k(x) + \mathsf{dL}'_k(z)$.
    
\item Both $x$ and $z$ are leaves. In this scenario, we have the equalities $\mathsf{dL}'_k(x) = \mathsf{dL}_{k-1}(x;y,z) + \mathsf{dL}_{k}(x,z;y) + \mathsf{dL}_{k-1}(x,z;y)$, $\mathsf{dL}'_k(y) = \mathsf{dL}_{k}(x,y;z) + \mathsf{dL}_{k}(y,z;x) + \mathsf{dL}_{k}(x,y,z;)$, and $\mathsf{dL}'_k(z) = \mathsf{dL}_{k}(x,z;y) + \mathsf{dL}_{k-1}(x;y,z) + \mathsf{dL}_{k-1}(x,z;y)$. Using our results from Case 2, it suffices to prove $\mathsf{dL}_{k}(x,y,z;) \leq \mathsf{dL}_{k-1}(x,z;y).$ This is true since for any set $S$ of $k-3$ leaves not including $x,y,$ or $z$, the Steiner tree of the set $S \cup \{x,z\}$ is the same as the Steiner tree of the set $S \cup \{x,y,z\}$. Thus, in this case, $2\mathsf{dL}'_k(x) \leq \mathsf{dL}'_k(y) + \mathsf{dL}'_k(z)$.
\end{enumerate}
Combining this casework with our initial consideration of sets of $k-1$ leaves not containing $x,y,$ or $z$ gives $2\mathsf{dL}_k(y) \leq \mathsf{dL}_k(x) + \mathsf{dL}_k(z)$.

\end{proof}

\begin{proposition}\label{Theorem: Steiner Internal Distance Concavity}
For a tree $T$ with vertices $x, y, z$ such that $xy, yz \in E(T)$ we have
\begin{equation*}
2\mathsf{dI}_k(y) \leq \mathsf{dI}_k(x) + \mathsf{dI}_k(z).
\end{equation*}
\end{proposition}
\begin{proof}
We first consider the Steiner distance contributed toward $\mathsf{dI}_k(x), \mathsf{dI}_k(y),$ and $\mathsf{dI}_k(z)$ by sets with one of $x, y, \textup{ or } z$. Denote $S_1$ as the set of sets of $k-1$ internal vertices not containing $x,y$ or $z$. Again, similar logic to that in Proposition \ref{Theorem: Steiner Distance Concavity} gives 
\begin{equation*}
    \sum_{t \in S_1}2\mathsf{dI}(y \cup S_1) \leq \sum_{t \in S_1} \left(\mathsf{dI}(x \cup S_1) + \mathsf{dI}(z \cup S_1) \right).
\end{equation*}

Now we account for sets of $k-1$ internal vertices containing at least one of $x,y$ or $z$. Let $\mathsf{dI}'_k(x)$ be the sum of Steiner distances of sets created by the union of $x$ and a set of $k-1$ internal vertices containing at least one of $x,y$ or $z$. Let $\mathsf{dI}'_k(y)$ and $\mathsf{dI}'_k(z)$ be defined similarly. Recall that for a vertex $v \in I(T)$, the set $\{v\} \cup S$ where $S$ is a set of $k-1$ internal vertices will have $k-1$ elements if $v \in S$. Also, because $y$ is adjacent to both $x$ and $z$, it is an internal vertex. We now use casework on whether $x$ and $z$ are internal vertices. 

\begin{enumerate} [label={Case }{{\arabic*}:}]
    \item Neither $x$ or $z$ are internal vertices. In this scenario, we have $\mathsf{dI}'_k(x) = \mathsf{dI}_k{(x,y;z)}$, $\mathsf{dI}'_k(y) = \mathsf{dI}_{k-1}{(y;x,z)}$, and $\mathsf{dI}'_k(z) = \mathsf{dI}_k{(yz;x)}$. Note that we must have $\mathsf{dI}'_k(y) \leq \mathsf{dI}'_k(x)$ in this case because for any set $S$ of $k-2$ vertices not including $x,y$ or $z$, the Steiner tree of $S \cup \{x,y\}$ will include the Steiner tree of $S \cup \{y\}$. Similarly, $\mathsf{dI}'_k(y) \leq \mathsf{dI}'_k(z)$. Thus, in this case, $2\mathsf{dI}'_k(y) \leq \mathsf{dI}'_k(x) + \mathsf{dI}'_k(z)$.
    
    \item Exactly one of $x$ or $z$ is an internal vertex. We first consider when $x$ is an internal vertex and $z$ is not. In this case, using known inequalities from Case 1 reduces proving $2\mathsf{dI}'_k(y) \leq \mathsf{dI}'_k(x) + \mathsf{dI}'_k(z)$ to showing $2(\mathsf{dI}_k{(x,y;z)} + \mathsf{dI}_{k-1}{(x,y;z)}) \leq \mathsf{dI}_{k-1}{(x,y;z)} + \mathsf{dI}_{k-1}{(x;y,z)} + \mathsf{dI}_k{(x,z;y)} + \mathsf{dI}_k{(x,y,z;)}.$ We will first prove that $\mathsf{dI}_{k-1}{(x,y;z)} \leq \mathsf{dI}_k{(x,y,z;)}$. This is true since for a set $S$ of $k-3$ internal vertices not containing $x,y,$ or $z$, the Steiner tree of $\{x,y\} \cup S$ will be included in the Steiner tree of $\{x,y,z\} \cup S$. Additionally, $2\mathsf{dI}_k{(x,y;z)} \leq \mathsf{dI}_{k-1}{(x;y,z)} + \mathsf{dI}_k{(x,z;y)}$ by a similar argument to that in Case 2 of Theorem \ref{Theorem: Steiner Leaf Distance Concavity}. This implies $2\mathsf{dI}'_k(y) \leq \mathsf{dI}'_k(x) + \mathsf{dI}'_k(z)$. The case where $z$ is an internal vertex but $x$ is not is symmetric.
    
    \item Both $x$ and $z$ are internal vertices. Using known inequalities from Case 1 and Case 2, we reduce proving $2\mathsf{dI}'_k(y) \leq \mathsf{dI}'_k(x) + \mathsf{dI}'_k(z)$ to showing $2(\mathsf{dI}_k{(x,y,z;)} + \mathsf{dI}_{k_1}{(x,y,z;)}) \leq \mathsf{dI}_{k-1}{(x,z;y)} + \mathsf{dI}_{k-1}{(x,y,z;)} + \mathsf{dI}_{k-1}{(x,z;y)} + \mathsf{dI}_{k-1}{(x,y,z;)}$ which reduces to $\mathsf{dI}_k{(x,y,z;)} \leq \mathsf{dI}_{k-1}{(x,z;y)}$. This is true since for a set $S$ of $k-3$ internal vertices not containing $x,y,$ or $z$, the Steiner tree of $\{x,z\} \cup S$ will be the same as the Steiner tree of $\{x,y,z\} \cup S$.
\end{enumerate}
Thus, by combining our casework analysis with our initial analysis on sets of $k-1$ internal vertices not containing $x,y,$ or $z$, we achieve our desired result.
\end{proof}

\section{Distances between the Steiner $k$-median and its variants}\label{Sec:centroids}

We start by establishing a criterion for Steiner $k$-centroids. For a tree $T$ and two vertices $x,y \in V(T)$ with $xy \in E(T)$, let $C_x$ and $C_y$ denote the connected components containing $x$ and $y$ respectively in $T-xy$. Let $T_x = C_x - x$ and $T_y = C_y - y$. We also define $n_{xy}(x)$ as the number of vertices $v$ with $\dd(v,x) < \dd(v,y)$. Note that $|V(T_x)| = n_{xy}(x) - 1$. The following result was shown in $\cite{berineke1996}$.
\begin{theorem}{\textup{\cite{berineke1996}}} \label{Theorem: Steiner distance comparison}
For a tree $T$ and two vertices $x,y \in V(T)$ with $xy \in E(T)$, 
\begin{align*}
\dd_k(x) = \dd_k(y) - \binom{n_{xy}(x)-1}{k-1} + \binom{n_{xy}(y)-1}{k-1}.
\end{align*}
\end{theorem}
When $x$ is a Steiner  $k$-centroid, we must have $\dd_k(x) \leq \dd_k(y)$, implying $\binom{n_{xy}(x)-1}{k-1} \geq \binom{n_{xy}(y)-1}{k-1}$. Thus, either $n_{xy}(x) \geq n_{xy}(y)$ or $n_{xy}(x), n_{xy}(y) < k$. Also, if $y \notin V(M_k(T))$, we have $n_{xy}(x) > n_{xy}(y)$ and $n_{xy}(x) \geq k$. This motivates the following definition.

\begin{definition}
A vertex $u$ is $k$-satisfactory if for all adjacent vertices $v$, either $n_{uv}(u) \geq n_{uv}(v)$ or $n_{uv}(u), n_{uv}(v) < k$. Also, if $v \notin V(M_k(T))$, we have $n_{uv}(u) > n_{uv}(v)$ with $n_{uv}(u) \geq k$.
\end{definition}

We now show the following result for completeness.

\begin{theorem}{\label{Theorem: Steiner centroid prerequisites}}
For a tree $T$, a vertex $v \in V(M_k(G))$ if and only if it is $k-$satisfactory.
\end{theorem}
\begin{proof}
By definition, every vertex $v \in V(M_k(G))$ is $k$-satisfactory. We now prove the opposite direction. Note that if two $k$-satisfactory vertices are adjacent, they must have equal Steiner $k$-distance. Now, let $u$ and $v$ be non-adjacent $k-$satisfactory vertices and let $P(u,v)$ be $u=u_0u_1\dots u_j=v$.
For all possible $i$, we have 
\begin{equation*}
n_{u_iu_{i+1}}(u_i) < n_{u_{i+1}u_{i+2}}(u_{i+1}) \textup{ and } n_{u_iu_{i+1}}(u_{i+1}) > n_{u_{i+1}u_{i+2}}(u_{i+2}).
\end{equation*}
Thus, we cannot have $n_{u_iu_{i+1}}(u_i) > n_{u_{i}u_{i+1}}(u_{i+1})$ with $n_{u_iu_{i+1}}(u_i) \geq k$ as that would contradict $v$ being $k$-satisfactory. We also cannot have $n_{u_iu_{i+1}}(u_{i}) < n_{u_{i}u_{i+1}}(u_{i+1})$ with $n_{u_{i}u_{i+1}}(u_{i+1}) \geq k$ as that would contradict $u$ being $k-$satisfactory.   

Thus, we conclude that $\dd_k(u) = \dd_k(u_1) = \dots = \dd_k(v).$  So, all $k-$satisfactory vertices have equal Steiner $k$-distance. This implies that all $k-$satisfactory vertices must be in the Steiner $k$-median. 
\end{proof}

Similar arguments lead to the following two lemmas.

\begin{lemma} \label{Theorem: Leaf-Centroid prerequisites}
For a tree $T$, a vertex $u \in V(T)$ is a Steiner $k$-leaf-centroid if and only if for all $v$ such that $uv \in E(T)$, either $|L'(T_u)| \geq |L'(T_v)|$ or $|L'(T_u)|, |L'(T_v)| < k - 1$. Also, if $v$ is not a Steiner $k$-leaf-centroid, we have $|L'(T_u)|  > |L'(T_v)|$ with $|L'(T_u)| \geq k-1$.
\end{lemma}
\begin{lemma} \label{Theorem: Internal-Centroid prerequisites}
For a tree $T$, a vertex $u \in V(T)$ is a Steiner $k$-internal-centroid if and only if for all $v$ such that $uv \in E(T)$, either $|I'(T_u)| \geq |I'(T_v)|$ or $|I'(T_u)|, |I'(T_v)| < k - 1$. Also, if $v$ is not a Steiner $k$-internal-centroid, we have $|I'(T_u)|  > |I'(T_v)|$ with $|I'(T_u)| \geq k-1$.
\end{lemma}

We now analyze distances between $M_k(T)$, $M^{\ell}_k(T)$, and $M^i_k(T)$.

\begin{theorem}\label{Theorem: Leaf-Internal Median Distance Bounds}
On a tree $T$, for the pair of vertices  $u\in V(M_k^{\ell }(T))$ and $v\in V(M_k^i(T))$ that minimizes $\dd(u, v)$,
\begin{equation*}\dd(u, v) \leq \max\{0, n + 1- \max \{k,3\} - \max\{k+1, \dd(u, v) + 2\}\}.
\end{equation*}
\end{theorem}

\begin{proof}
Consider a tree $T$ that maximizes the distance between the two medians where the distance is defined as $\dd(u,v)$ with $u\in V(M_k^{\ell }(T))$ and $v\in V(M_k^i(T))$ such that $\dd(u,v)$ is minimized. If $u = v$, then $\dd(u, v) = 0$. Else, let $u'$ and $v'$ be the adjacent vertices of $u$ and $v$ respectively on $P(u, v)$ and let $C_u$ and $C_v$ be the connected components containing $u$ and $v$ respectively in $T-E(P(u,v))$.

Note that since $u' \notin V(M_k^{\ell}(T))$, by Lemma \ref{Theorem: Leaf-Centroid prerequisites}, 
\begin{equation*}
    |L'(C_u - u)| \geq k-1 \textup{ and }  |L'(C_u - u)| > |L'(T-C_u-u')|.
\end{equation*}
Note that $|L'(T-C_u-u')| \geq 1$ as otherwise $u' \in L(T)$, implying $u=v$. Thus, $V(C_u) \geq \max \{k,3\}$.

Additionally since $v' \notin V(M_k^{i}(T))$, by Lemma \ref{Theorem: Internal-Centroid prerequisites}, 
\begin{equation*}
    |I'(C_v-v)| \geq k-1 \textup{ and }  |I'(C_v -v)| >  |I'(T-C_v-v') | \geq \dd(u, v)-1.
\end{equation*}
Thus, $|V(C_v)| \geq \max \{k+1, \dd(u,v) + 2\}$.

Thus we have 
\begin{align*}
    n &= |V(T)|\\
     &\geq |V(C_u)| + |V(C_v)| + \dd(u, v) - 1\\
      &\geq \max \{k,3\} + \max \{ k+1, \dd(u, v) + 2 \} + \dd(u, v) - 1.
\end{align*}
Thus, 
\begin{equation*}
     \dd(u, v) \leq n + 1 - \max \{k,3\} - \max\{k+1, \dd(u, v) + 2\}.
\end{equation*}
This concludes our proof. 
\end{proof}
\begin{theorem} \label{Theorem: Leaf-Normal Median Distance Bound}
For a tree $T$ and $k \leq \frac{n+1}{2}$, the minimal distance between a vertex  $u\in V(M_k^{\ell }(T))$ and a vertex $v\in V(M_k(T))$ is bounded
\begin{equation*}
\dd(u, v) \leq \frac{n-2\max \{k,3\}+1}{2}.
\end{equation*}
\end{theorem}
\begin{proof}
Consider a tree $T$ that maximizes the distance between the two medians, where the distance is defined as $\dd(u,v)$ with $u\in V(M_k^{\ell }(T))$ and $v\in V(M_k(T))$ such that $\dd(u,v)$ is minimized. If $u = v$, then $\dd(u, v) = 0$. Else, let $u'$ and $v'$ respectively be the adjacent vertices of $u$ and $v$ on $P(u, v)$. Also let $C_u$ and $C_v$ be the connected components containing $u$ and $v$ respectively in $T-E(P(u,v))$.

From Theorem \ref{Theorem: Leaf-Internal Median Distance Bounds} we have, $  |V(C_u)| \geq \max \{k,3\}$. Additionally by Lemma \ref{Theorem: Steiner centroid prerequisites},
\begin{equation*}
    |V(C_v)| > |V(T-C_v)| \geq |V(C_u)| + \dd(u,v) - 1 \geq \max \{k,3\} + \dd(u,v) - 1.  
\end{equation*}
Thus, \
\begin{equation*}
    n \geq |V(C_u)|+|V(C_v)| + \dd(u,v) -1 \geq 2\max \{k,3\} + 2\dd(u,v)-1,
\end{equation*}
and the result follows.
\end{proof}
\begin{theorem}
For a tree $T$ and $k \leq \frac{n-1}{2}$, the minimum distance between a vertex  $u\in V(M_k^{i}(T))$ and a vertex $v\in V(M_k(T))$ is bounded
\begin{equation*}
  \dd(u,v) \leq \frac{n - 2 \max \{ k+1, \dd(u, v) + 2 \} + 1}{2}.
\end{equation*}
\end{theorem}
\begin{proof}
Consider a tree $T$ that maximizes the distance between two medians, where the distance is defined as $\dd(u,v)$ with $ u \in V(M_k^{i}(T))$ and $v\in V(M_k(T))$ such that $\dd(u,v)$ is minimized. If $u = v$, then $\dd(u, v) = 0$. Else, let $u'$ and $v'$ respectively be the adjacent vertices of $u$ and $v$ on $P(u, v)$. Also let $C_u$ and $C_v$ be the connected components containing $u$ and $v$ respectively in $T-E(P(u,v))$.

From Theorem \ref{Theorem: Leaf-Internal Median Distance Bounds} we have 
\begin{equation*}
    |V(C_u)| \geq \max \{ k + 1, \dd(u, v) + 2 \}.
\end{equation*}
Additionally by Lemma \ref{Theorem: Steiner centroid prerequisites}, 
\begin{equation*}
    |V(C_v)| > |V(T-C_v)| \geq |V(C_u)| + \dd(u,v) -1\geq \max \{k+1, \dd(u, v) + 2 \} + \dd(u, v)-1.  
\end{equation*}
Thus, 
\begin{equation*}
    n \geq |V(C_u)|+|V(C_v)| + \dd(u,v) - 1 \geq 2 \max \{ k+1, \dd(u, v) + 2 \} + 2\dd(u, v) - 1.
\end{equation*}
and the result follows.
\end{proof}

\begin{remark}
The bounds provided in this section are not sharp. The specific conditions for equalities usually depend on the value of $k$, and they involve rather tedious cases. We chose to omit the technical details for this reason.
\end{remark}

\section{Extremal Ratios}\label{Sec:ratio}
In this section we consider ratios between the Steiner $k$-distances of specific vertices and between the Steiner $k$-distance of a Steiner $k$-centroid and the Steiner $k$-Wiener index. For a vertex $v$ within a tree $T$, we define the branches at $v$ as the subtrees of $T$ with $v$ as a leaf. 

\subsection{Extremal ratios between similar vertices}

\begin{theorem}\label{Theorem: Bounds for ratio of leafs}
Let $F = 1 +  \frac{\binom{n-2}{k} - \binom{n-r}{k} - \binom{r-1}{k}}{\binom{n-r}{k} + (n-r-1)\binom{n-2}{k-2} + r\binom{n-1}{k-1} - \binom{n-1}{k}}$ where $1 \leq r \leq n-1$ is an integer that maximizes $F$. In a tree $T$ on $n$ vertices with $w,v \in L(T)$, we have
\begin{equation*}
   \frac{1}{F} \leq \frac{\dd_k(v)}{\dd_k(w)} \leq F.
\end{equation*}
 The maximum bound is achieved when $T$ is a $r$-comet where $w = a$ and $v=b$ as in Definition \ref{Def: r-comet and vertices}, and the minimum bound is achieved when $w$ and $v$ switch places.
\end{theorem}
\begin{proof}
Note that it suffices to prove the maximum bound as the minimum follows by swapping the positions of $w$ and $v$. Let $T$ be a tree with $w,v \in L(T)$  that maximizes $\frac{\dd_k(v)}{\dd_k(w)}$. First, denote the path $w = w_0w_1 \dots w_r = v$ as $P$. Under our prior restrictions, let $T$ maximize the degree of $w_1$. We  claim that all vertices not in $P$ must lie in a branch of $w_1$ that contains no vertices in $P$ besides $w_1$. Assume for contradiction that there is a vertex $q$ not in $P$ that is adjacent to $w_i$ where $i \neq 1$. Then let 
$T' = T - w_iq+w_1q$. Note that 
\begin{equation*}
\dd^T_k(w) \geq \dd^{T'}_k(w) \textup{ and }\dd^T_k(v) \leq \dd^{T'}_k(v),  \end{equation*}
a contradiction as $w_1$ has higher degree in $T'$ and $\frac{\dd^{T}_k(v)}{\dd^{T}_k(w)} \leq \frac{\dd^{T'}_k(v)}{\dd^{T'}_k(w)}$. 

Now, since all vertices in $T$ either lie in $P$ or a branch of $w_1$ with no vertices of $P$ besides $w_1$, we have $\dd_k(w_0) = \dd_k(w_1) + \binom{n-2}{k-1},$ and for $i \geq 0$, $ \dd_k(w_{i+2}) = \dd_k(w_{i+1}) + \binom{n-r+i}{k-1} - \binom{r-i-2}{k-1}. $
Thus, 
\begin{align*}
\dd_k(w_r)& = \dd_k(w_1) + \binom{n-r}{k-1} + \binom{n-r+1}{k-1} + \dots + \binom{n-2}{k-1} \\
 & \quad\quad\quad - \binom{r-2}{k-1} - \binom{r-3}{k-1} - \dots - \binom{k-1}{k-1}\\
& = \dd_k(w_1) + \binom{n-1}{k} - \binom{n-r}{k} - \binom{r-1}{k}\\
& = \dd_k(w_0) + \binom{n-2}{k} - \binom{n-r}{k} - \binom{r-1}{k}.
\end{align*}
So,  
\begin{equation*}
\frac{\dd_k(v)}{\dd_k(w)} = 1 +  \frac{1}{\dd_k(w)}
\left(\binom{n-2}{k} - \binom{n-r}{k} - \binom{r-1}{k}\right).
\end{equation*}
So, as we must minimize $\dd_k(w)$, we note that all vertices adjacent $w_1$ not on $P$ must have degree 1, giving us Figure \ref{figure41}.
\begin{figure}[hbt]
\begin{center}
\begin{tikzpicture}[scale = 0.7]
\draw (1,-1) -- (3,-1)
 node [draw=black, circle, fill=black, inner sep=2pt, pos=0]  {}
  node [draw=black, circle, fill=black, inner sep=2pt, pos=1]  {};
\draw  (3,-1) --(5, -1) 
    node [draw=black, circle, fill=black, inner sep=2pt, pos=0]  {};
 \draw [dashed](5,-1) -- (8,-1) 
    node [draw=black, circle, fill=black, inner sep=2pt, pos=0]  {}
  node [draw=black, circle, fill=black, inner sep=2pt, pos=1]  {};
   \draw (8,-1) -- (10,-1)
     node [draw=black, circle, fill=black, inner sep=2pt, pos=0]  {}
  node [draw=black, circle, fill=black, inner sep=2pt, pos=1]  {};
 
\draw (1.5,-3) -- (3,-1) 
 node [draw=black, circle, fill=black, inner sep=2pt, pos=0]  {};
\draw (2.5,-3) -- (3,-1) 
node [draw=black, circle, fill=black, inner sep=2pt, pos=0]  {};
 %\draw [loosely dashed](2.5,-3) -- (5,-3) ;
\draw  (5, -3) -- (3,-1)
 node [draw=black, circle, fill=black, inner sep=2pt, pos=1]  {}
  node [draw=black, circle, fill=black, inner sep=2pt, pos=0]  {};

% \draw (1,-5) -- (0,-6.5) -- (2,-6.5) -- (1,-5);
\node [right] at (2.8, -3) {$\large \ldots \ldots $};
  \node [above] at (1, -1) {$w$};
    \node [above] at (3, -1) {$w_1$};
      \node [above] at (5, -1) {$w_2$};
        \node [above] at (8, -1) {$w_{r-1}$};
  \node [above] at (10, -1) {$v$};
  
\end{tikzpicture}
\end{center}
\caption{A graph maximizing $\frac{\dd_k(v)}{\dd_k(w)}$}
\label{figure41}
\end{figure}
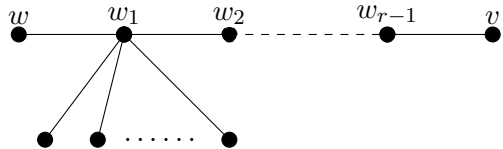

By casework on the number of leaves adjacent to $w_1$ chosen to be in a set of $k-1$ vertices that does not include $w$, we have
\begin{align*}
    \dd_k(w) & = \binom{n-r-1}{k-1}(k) + \sum_{i = 0}^{k-2}\sum_{j = 0}^{r - 1- i} \binom{n-r-1}{k-2-i}(r + k -2 -i-j)\binom{r - 1-j}{i}.
\end{align*}

\begin{comment}
\begin{align*}
    \dd_k(w) & = \binom{n-r-1}{k-1}(k)\\
    &+  \binom{n-r-1}{k-2}\left((r+k-2)\binom{r-1}{0} + (r+k-3)\binom{r-2}{0} + \cdots + (k-1)\binom{0}{0} \right)\\
    &+  \binom{n-r-1}{k-3}\left((r+k-3)\binom{r-1}{1} + (r+k-4)\binom{r-2}{1} + \cdots + (k-1)\binom{1}{1} \right)\\
    &+ \cdots \\
    &+ \binom{n-r-1}{0}\left((r)\binom{r-1}{k-2} + (r-1)\binom{r-2}{k-2} + \cdots + (k-1)\binom{k-2}{k-2}\right).
\end{align*}
\end{comment}

Simplifying using the Hockey Stick Identity and rearranging, we get
\begin{align*}
    \dd_k(w) = &\binom{n-r-1}{k-1} + \binom{n-r-1}{k-1}\left((r+k-1)\binom{r}{0} - \binom{r}{1}\right)
   \\
   & \quad\quad\quad +  \binom{n-r-1}{k-2}\left((r+k-2)\binom{r}{1} - \binom{r}{2}\right)\\ 
    & \quad\quad\quad + \binom{n-r-2}{k-3}\left((r+k-3)\binom
    {r}{2} - \binom{r}{3}\right)  \\
     & \quad\quad\quad \cdots +  \binom{n-r-1}{0}\left((r)\binom{r}{k-1} - \binom{r}{k}\right).
\end{align*}

Applying Vandermonde's Identity and simplifying, we have
\begin{equation*}
  \dd_k(w) = \binom{n-r}{k} + (n-r-1)\binom{n-2}{k-2} + r\binom{n-1}{k-1} - \binom{n-1}{k},
\end{equation*}
and our proof is complete.
\end{proof}

\begin{theorem} \label{Theorem: Bounds for internal vertices}
Let $F = 1 +  \frac{\binom{n-2}{k} - \binom{n -r - 2}{k} -\binom{r + 1}{k}}{ (n-r-2)\binom{n-2}{k-2} + (r+1)\binom{n-1}{k-1}
+ \binom{n-r-2}{k}
- \binom{n-1}{k}}$, where $1 \leq r \leq n-2$ is an arbitrary positive integer that maximizes $F$. In a tree $T$ with $w,v \in I(T)$, we have 
\begin{equation*}
   \frac{1}{F} \leq \frac{\dd_k(v)}{\dd_k(w)} \leq F.
\end{equation*}
The maximum bound is achieved when $T$ is a $(r+2)-$comet where $w=d$ and $v=c$ as in Definition \ref{Def: r-comet and vertices}. The minimum bound is achieved when $w$ and $v$ are swapped.
\end{theorem}

\begin{proof}
Note that it suffices to find the maximum value of $\frac{\dd_k(v)}{\dd_k(w)}$ as we can swap $w$ and $v$ to obtain the minimum. Let $T$ be a tree with $w,v \in I(T)$ maximizing $\frac{\dd_k(v)}{\dd_k(w)}$. Denote the path $w = w_0w_1 \dots w_r = v$ as $P$. Set $T$ to be the tree that maximizes the degree of $w$ under the prior restrictions and let $C_w$ be the connected component containing $w$ in $T-w_0w_1$. Note that since $v$ is an internal vertex, we must have a vertex $w_{r+1}$ not on $P$ that is adjacent to $v$. By using similar logic to Theorem \ref{Theorem: Bounds for ratio of leafs}, we have that all other vertices either lie in $P$ or $C_w$ as shown in Figure \ref{fig:internal-internal-ratio}. 
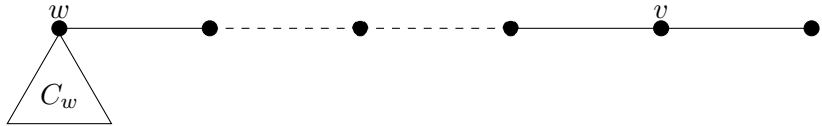
\begin{figure}[htb]
\begin{center}
\begin{tikzpicture}

 \draw (2,0) -- (4,0) 
  node [draw=black, circle, fill=black, inner sep=2pt, pos=0] {}
 node [draw=black, circle, fill=black, inner sep=2pt, pos=1]  {};
\node [above] at (2,0) {$w$};
 \draw [dashed](4,0) -- (6,0) 
  node [draw=black, circle, fill=black, inner sep=2pt, pos=1]  {};
 \draw [dashed](6,0) -- (8,0) 
 node [draw=black, circle, fill=black, inner sep=2pt, pos=1]  {};
 \draw (8,0) -- (10,0) 
 node [draw=black, circle, fill=black, inner sep=2pt, pos=1]  {};
  \draw (10,0) -- (12,0) 
 node [draw=black, circle, fill=black, inner sep=2pt, pos=1]  {};
 \node [above] at (10,0) {$v$};
 \node[isosceles triangle,isosceles triangle apex angle=60,
    draw,
    rotate=90,
    minimum size =1.2cm] (T1)at (2,-0.9){};

% \draw (2,0) -- (1,-2) -- (3,-2) -- (2,0);
%  \draw (10,0) -- (9,-2) -- (11,-2) -- (10,0);
\node [above] at (2,-1.2) {$C_w$};
\end{tikzpicture}
\end{center}
\caption{A tree maximizing $\frac{\dd_k(v)}{\dd_k(w)}$}
\label{fig:internal-internal-ratio}
\end{figure}
Note that for all $i\geq 0$,
\begin{equation*}
\dd_k(w_{i+1}) = \dd_k(w_i) +\binom{n -r - 2+i}{k-1} - \binom{r-i}{k-1}.
\end{equation*}
Thus, 
\begin{align*}
    \dd_k(v) &= \dd_k(w) + \binom{n-r -2}{k-1} + \binom{n-r -1}{k-1} + \cdots + \binom{n-3}{k-1} \\
    &-\binom{r}{k-1} -\binom{r-1}{k-1} - \cdots - \binom{k-1}{k-1}.
\end{align*}
Simplifying gives $\dd_k(v) = \dd_k(w) + \binom{n-2}{k} - \binom{n -r - 2}{k} -\binom{r + 1}{k}.$
So, we have 
\begin{equation*}
    \frac{\dd_k(v)}{\dd_k(w)} = 1 + \frac{\binom{n-2}{k} - \binom{n -r - 2}{k} -\binom{r + 1}{k}}{\dd_k(w)}.
\end{equation*}
To minimize $\dd_k(w),$ all vertices adjacent to $w$ except $w_1$ must be leaves. Now, casework on the number of leaves adjacent to $w$ chosen to be in a set of $k-1$ vertices that does not include $w$ gives the formula
\begin{align*}
    \dd_k(w) & = \binom{n-r-2}{k-1}(k-1) + \sum_{i = 0}^{k -2}\binom{n-r-2}{k-2-i}\sum_{j = 0}^{r - i}(r + k -1 -i - j) \binom{r - j}{i}.
\end{align*}

\begin{comment}
\begin{align*}
    \dd_k(w) & = \binom{n-r-2}{k-1}(k-1)\\
    &+  \binom{n-r-2}{k-2}\left((r+k-1)\binom{r}{0} + (r+k-2)\binom{r-1}{0} + \cdots + (k-1)\binom{0}{0} \right)\\
  %  &+  \binom{n-r-2}{k-3}\left((r+k-2)\binom{r}{1} + (r+k-3)\binom{r-1}{1} + %\cdots + (k-1)\binom{1}{1} \right)\\
    &+ \cdots \\
    &+ \binom{n-r-2}{0}\left((r +1)\binom{r}{k-2} + (r)\binom{r-1}{k-2} + \cdots + (k-1)\binom{k-2}{k-2}\right).
\end{align*}
\end{comment}

\begin{comment}
\begin{align*}
    \binom{n-r-2}{k-1}&\left((r+k)\binom{r+1}{0} - \binom{r+1}{1}\right)\\
    +  \binom{n-r-2}{k-2}&\left((r+k-1)\binom{r+1}{1} - \binom{r+1}{2}\right)\\ 
    +   \binom{n-r-2}{k-3}&\left((r+k-2)\binom
    {r+1}{2} - \binom{r+1}{3}\right)\\
    + \cdots \\
    +  \binom{n-r-2}{0}&\left((r+1)\binom{r+1}{k-1} - \binom{r+1}{k}\right).
\end{align*}
\end{comment}

Simplifying as in Theorem \ref{Theorem: Bounds for ratio of leafs}, we get 
\begin{equation*}
\dd_{k}(w) = (n-r-2)\binom{n-2}{k-2} + (r+1)\binom{n-1}{k-1}
+ \binom{n-r-2}{k}
- \binom{n-1}{k}.
\end{equation*}
So, our proof is complete.
\end{proof}

\subsection{Extremal ratios between unlike vertices}
We start with two observations which will be used to prove Theorem \ref{Theorem: Bounds for one leaf and one centroid}. 
\begin{lemma} \label{Theorem: Branches are paths to maximize Steiner distance}
Given a tree $T$ with leaf $w \in L(T)$ and vertex $v \in V(T)$ such that $vw \in E(T)$, let $C_v$ be an arbitrary branch at $v$ and let $S = T - (C_v-v)$ be fixed. For $\dd_k(v)$ to be maximized, $C_v$ must be a path.
\end{lemma}
\begin{proof}
Note that $\dd^{T}_k(v) = \dd_{k}^{S}(v) + \dd_{k}^{C_v}(v) + \sum_{i=2}^{k-1}\dd_i^{S}(v)\dd_{k+1-i}^{C_v}(v)$. Thus, to maximize $\dd_k^{T}(v)$, we must maximize $\dd_{i}^{C_v}(v)$ for integers $2 \leq i \leq k$. Note that setting $C_v$ to be a path achieves this since for all $k-1 \leq j \leq n-1$, a path has $\binom{n-1}{k-1}-\binom{j-1}{k-1}$ sets of $k-1$ vertices not including $v$ such that the Steiner distance of the union of $v$ and that set is at least $j$. This is maximal.
\end{proof}

\begin{lemma}\label{Prep: Moving end vertices of paths}
Take a tree $T$ with $v \in V(M_k(T))$ and $w \in L(T)$ such that $vw \in E(T)$. Given two branches $P_1$ and $P_2$ of $v$ that do not include $w$ and are paths such that $|V(P_1)| \leq |V(P_2)|$, let $P_1$ be the path $v = v_0v_1\dots v_a$ and $P_2$ be the path $v = v_0r_1 \dots r_b$. Then, we call a leaf-transfer from $P_1$ to $P_2$ as creating another tree $T' = T-v_av_{a-1}+v_ar_b$. This gives $d_k^{T}(v) \leq d_k^{T'}(v)$. 
\end{lemma}
\begin{proof}
Let $C_v$ be the connected component containing $v$ in $T-v_0v_1-v_0r_1$. We now define two more graphs, $S = T - (C_v-v)$ and $S' = T' - (C_v-v)$, as shown in Figure \ref{fig:SandS}. Note that $d_k^{T}(v) \leq d_k^{T'}(v)$ if and only if $d_k^{S}(v) \leq d_k^{S'}(v)$. The latter inequality is true by noting $d_k^{S}(v) = d_k^{S'}(r_1)$ and applying Theorem \ref{Theorem: Steiner distance comparison}, so the first is as well.
\end{proof}

\begin{figure}[htb] 
\begin{center}
\begin{tikzpicture}[scale = 0.7]
\draw (0, 1) -- (2,1);
 \draw [dashed](2,1) -- (4,1) ;
\draw (4, 1) -- (6,1) --(8,1);
 \draw [dashed](8,1) -- (10,1) ;

\fill (0, 1) circle (2pt) node[] {};
\fill (2, 1) circle (2pt) node[] {};
\fill (4, 1) circle (2pt) node[] {};
\fill (6, 1) circle (2pt) node[] {};
\fill (8, 1) circle (2pt) node[] {};
\fill (10, 1) circle (2pt) node[] {};

 \node [below] at (0, 1) {$v_a$};
  \node [below] at (2, 1) {$v_{a-1}$};
   \node [below] at (4, 1) {$v_1$};
  \node [below] at (6, 1) {$v_0$};
   \node [below] at (8, 1) {$r_1$};
  \node [below] at (10, 1) {$r_b$};
  \node [left] at (-1, 1) { $S$};
  
  % \node[isosceles triangle,isosceles triangle apex angle=60,
  %  draw,
  % rotate=90,
  %  minimum size = 1.1cm] (T1)at (6, -0){};
  
  \draw[dashed] (0, -1.5) -- (2,-1.5);
 \draw(2,-1.5) -- (4,-1.5)-- (6,-1.5) ;
\draw[dashed] (6,-1.5) --(8,-1.5);
 \draw (8,-1.5) -- (10,-1.5);

\fill (0, -1.5) circle (2pt) node[] {};
\fill (2, -1.5) circle (2pt) node[] {};
\fill (4, -1.5) circle (2pt) node[] {};
\fill (6, -1.5) circle (2pt) node[] {};
\fill (8, -1.5) circle (2pt) node[] {};
\fill (10, -1.5) circle (2pt) node[] {};

 \node [below] at (0, -1.5) {$v_{a-1}$};
  \node [below] at (2, -1.5) {$v_1$};
   \node [below] at (4, -1.5) {$v_0$};
  \node [below] at (6, -1.5) {$r_1$};
   \node [below] at (8, -1.5) {$r_b$};
  \node [below] at (10, -1.5) {$v_a$};
  \node [left] at (-1, -1.5) {$ {S'}$};
  
 %  \node[isosceles triangle,isosceles triangle apex angle=60,
 %   draw,
 %   rotate=90,
 %   minimum size = 1.1cm] (T1)at (4, -2.5){};
    
\end{tikzpicture}
\end{center}
 \caption{$S \text{ } and \text{ }S'$}
 \label{fig:SandS}
\end{figure}

We also state the following lemmas which will help with final calculations in Theorem \ref{Theorem: Bounds for one leaf and one centroid}.
\begin{lemma}\label{Theorem: Steiner distance of a vertex on an a+b+1 path}
Given a path $P$ of length $a+b+1$, let $v$ be a vertex on $P$ of distance $a$ from one leaf on $P$ and distance $b$ from the other. Then, $\dd_k(v) = (k-1)\binom{b+a +1}{k}- (k-1) \binom{b+1}{k} -(k-1)\binom{a+1}{k} -\binom{b+a}{k} + b\binom{b}{k-1} + a\binom{a}{k-1}$.
\end{lemma}
\begin{proof}
Let $P_1$ denote the subgraph of $P$ containing the $a$ vertices on one side of $v$ and let $P_2$ denote the subgraph of $P$ containing the $b$ vertices on the other side of $v$. By the definition of Steiner $k$-distance, to find $\dd_k(v),$ we may use casework on the location of the $k-1$ vertices chosen to accompany $v$. The sum of the Steiner distances of all sets of $k$ vertices where $k-1$ vertices are in $P_2$ and the last is $v$ is
\begin{equation*}
    \sum_{i=1}^{b}\binom{b-i}{k-2}(b+1-i)
    = \sum_{i=1}^{b}\left(\binom{b}{k-1} - \binom{i-1}{k-1}\right)
    = b\binom{b}{k-1}-\binom{b}{k}.
\end{equation*}
Similarly, the sum of the Steiner distance of all sets of $k$ vertices where $k-1$ vertices are in $P_1$ and the last is $v$ is $a\binom{a}{k-1}-\binom{a}{k}$. The sum of the Steiner distance of all sets of $k$ vertices including $v$ with vertices in both $P_1$ and $P_2$ is, by casewo rk on the distance of the furthest vertex from $v$ chosen in $P_1$,
\begin{align*}
  &\sum_{i=1}^{a}\sum_{j=1}^{b}\binom{i-1+b-j}{k-3}(i+1+b-j) \\
    & \quad = \sum_{i = 1}^a \left( ( b+ i)\binom{b+ i -2}{k -3} + ( b+ i - 1)\binom{b+ i -3}{k -3} + \dots + (i + 1)\binom{i -1}{k - 3} \right) \\
    \end{align*}
    Now, applying the Hockey Stick Identity gives
    \begin{align*}
  %  & \quad = \sum_{i = 1}^a \left( (i + 1) \left( \binom{b+ i -1}{k -2} - \binom{i-1}{k -2} \right)  + ( b - 1)\binom{b+ i -1}{k -2} - \binom{i}{k-2} - \cdots - \binom{b+ i - 2}{k-2} \right) \\
    & \quad = \sum_{i = 1}^a \left( (b+i) \binom{b+ i - 1}{k -2} -(i+1)\binom{i-1}{k-2} -\binom{b+i-1}{k-1} + \binom{i}{k-1} \right) \\
    & \quad = (k-1) \sum_{i = 1}^a \binom{b+i}{k-1} + \sum_{i = 1}^a \binom{i}{k-1} - \sum_{i = 1}^a \binom{b+i-1}{k-1} - \sum_{i = 1}^a (i+1)\binom{i-1}{k-2} \\
   % & \quad = (k-1)\left( \binom{b+a +1}{k}- \binom{b+1}{k} \right)+ \binom{a+1}{k} - \left( \binom{b+a}{k}- \binom{b}{k} \right) -\sum_{i = 1}^a (i+1)\binom{i-1}{k-2} \\
   % & \quad =(k-1)\left( \binom{b+a +1}{k}- \binom{b+1}{k} \right) + \binom{a+1}{k} - \left( \binom{b+a}{k}- \binom{b}{k} \right) -(a+1)\binom{a}{k-1} + \binom{a}{k}\\
    & \quad =(k-1)\binom{b+a +1}{k}- (k-1) \binom{b+1}{k} -(k-1)\binom{a+1}{k} + \binom{a}{k} + \binom{b}{k} -\binom{b+a}{k}.
\end{align*}
Summing over all cases gives our desired result.
\end{proof}
\begin{lemma}\label{Theorem: a+b=1 path with vertex Steiner distance}
Let $T$ be a graph consisting of a path $P$ of order $a+b+1$ and a vertex $w \notin V(P)$ such that $w$ is adjacent to a vertex $v \in V(P)$ such that $a$ vertices in $P$ are on one side of $v$ and $b$ vertices are on the other. Then, 
\begin{align*}
    \dd_k(v) & = (k-2)\binom{b+a+2}{k}-(k-2-b)\binom{b+1}{k-1}-(k-2-a)\binom{a+1}{k-1}-(k-1)\binom{b+1}{k}\\
    &\quad -(k-1)\binom{a+1}{k}+ \binom{a+b}{k-2}.
\end{align*}
\end{lemma}
\begin{proof}
To find $\dd_k(v)$, we note that in any set $S$ of $k$ vertices with $v \in S$, either $w \in S$ or $w \notin S$. Finding the sum of the Steiner distances of the sets where $w \notin S$ is equivalent to $\dd_k^{T'}(v)$ where $T'$ is the path $P$. By Lemma \ref{Theorem: Steiner distance of a vertex on an a+b+1 path}, $\dd_k^{T'}(v) = (k-1)\binom{b+a +1}{k}- (k-1) \binom{b+1}{k} -(k-1)\binom{a+1}{k} -\binom{b+a}{k} + b\binom{b}{k-1} + a\binom{a}{k-1}$. Also, the sum of the Steiner distances of the sets where $w \in S$ is equivalent to $\dd_{k-1}^{T'}(v) + \binom{a+b}{k-2}$ since each such set $S$ consists of $w$, which is of distance $1$ to $v$, and $k-2$ vertices on $P$ besides $v$. By Lemma \ref{Theorem: Steiner distance of a vertex on an a+b+1 path}, this is $(k-2)\binom{b+a+1}{k-1}- (k-2) \binom{b+1}{k-1} -(k-2)\binom{a+1}{k-1} -\binom{b+a}{k-1} + b\binom{b}{k-2} + a\binom{a}{k-2} + \binom{a+b}{k-2}$. Summing these together gives
\begin{align*}
    \dd_k(v) & = (k-2)\left( \binom{b+a+1}{k-1} - \binom{b+1}{k-1}  - \binom{a+1}{k-1} \right) \\
    & \quad\quad\quad + (k-1) \left(\binom{b+a +1}{k} -  \binom{b+1}{k}  - \binom{a+1}{k} \right)\\
    &\quad\quad\quad  -\binom{b+a+1}{k} + b\binom{b+1}{k-1} + a\binom{a+1}{k-1} + \binom{a+b}{k-2},
\end{align*}
which simplifies into our desired result.
\end{proof}

For convenience, we also define the following function.
\begin{align*}
F(a,b) &= (k-2)\binom{b+a+2}{k}-(k-2-b)\binom{b+1}{k-1}-(k-2-a)\binom{a+1}{k-1}-(k-1)\binom{b+1}{k}\\
&\quad -(k-1)\binom{a+1}{k}+ \binom{a+b}{k-2}.    
\end{align*}

\begin{theorem} \label{Theorem: Bounds for one leaf and one centroid}
A tree $T$ of order $n$ with $w \in L(T)$ and $v \in V(M_k(T))$ satisfies the following.
\begin{itemize}
    \item 
    When $k \geq n-1$ the graph minimizing $\frac{\dd_k(w)}{\dd_k(v)}$ is a path of length $n-1$ with $w$ as a leaf and $v$ as a vertex adjacent to $w$. Also, $\frac{\dd_k(w)}{\dd_k(v)} \geq 1+\frac{\binom{n-2}{k-1}}{n-1}$ when $k=n$ and $\frac{\dd_k(w)}{\dd_k(v)} \geq 1 + \frac{\binom{n-2}{k-1}}{2(n-2) + (n-3)(n-1)}$ when $k=n-1$. 
    \item
    When $n$ is even and $k \leq \frac{n}{2} + 1$, the graph minimizing $\frac{\dd_k(w)}{\dd_k(v)}$ consists of a path of length $n-2$ where $v$ is a vertex on the path of distance $\frac{n}{2}$ from a leaf of the path and a vertex $w$ not on the path that is adjacent to $v$. Also, 
\begin{equation*}
    \frac{\dd_k(w)}{\dd_k(v)} \geq 1+ \frac{\binom{n-2}{k-1}}{F(\frac{n}{2},\frac{n}{2}-2)}.
\end{equation*}
\item
 
When $n$ is odd and $k \leq \frac{n}{2} + 1$, the graph minimizing $\frac{\dd_k(w)}{\dd_k(v)}$ consists of a path of length $n-2$ where $v$ is a vertex on the path of distance $\frac{n-1}{2}$ from a leaf on the path and a vertex $w$ not on the path that is adjacent to $v$. Also,
 \begin{equation*}
     \frac{\dd_k(w)}{\dd_k(v)} \geq 1+ \frac{ \binom{n-2}{k-1}}{F(\frac{n-1}{2},\frac{n-3}{2})}.
 \end{equation*}

 \item
  When $k > \frac{n}{2} + 1$, the graph minimizing $\frac{\dd_k(w)}{\dd_k(v)}$ consists of a path of length $n-2$ where $v$ is a vertex on the path of distance $k-1$ from a leaf on the path and a vertex $w$ not on the path that is adjacent to $v$. Also,
 \begin{equation*}
 \frac{\dd_k(w)}{\dd_k(v)} \geq 1+\frac{ \binom{n-2}{k-1}}{F(k-1, n-k-1)}.
 \end{equation*}
\end{itemize}
 
\end{theorem}
\begin{proof}
First note that in a tree $S$ of order $n$ with $w \in L(S)$ and $v \in V(M_k(S))$ such that $vw \in E(S)$, we have that 
\begin{equation*}
    \frac{\dd^S_k(w)}{\dd^S_k(v)} = 1 + \frac{ \binom{n-2}{k-1}}{\dd^S_k(v)}.
\end{equation*}

Additionally, since $\dd_k(v) \geq \binom{n-1}{k-1}(k-1)$, we know that
\begin{equation*}
    \frac{\dd^S_k(w)}{\dd^S_k(v)} \leq 1 + \frac{\binom{n-2}{k-1}}{\binom{n-1}{k-1}(k-1)}.
\end{equation*}

Now, consider an arbitrary tree $T$ that minimizes $\frac{\dd_k(w)}{\dd_k(v)}$ over all trees of order $n$. Since all elements in $M_k(T)$ have equal Steiner $k$-distance, we let $v$ be the element in $M_k(T)$ with least distance to $w$.

Assume for contradiction that $vw \notin E(T)$ and note that there is then a unique path $w = w_0w_1 \dots w_r = v$. From similar logic to Theorem \ref{Theorem: Bounds for ratio of leafs}, $deg(w_i) = 2$ for $2 \leq i \leq r-1$. Let the connected component containing $w_1$ in $T-w_0w_1-w_1w_2$ be denoted by $C_{w_1}$. Set $B=|V(C_{w_1})|$. Then, note that by Theorem \ref{Theorem: Steiner centroid prerequisites}, $v$ is still a Steiner $k-$centroid in the tree $T' = T-w_0w_1+w_2w_0.$ Additionally,
\begin{align*}
\dd^{T'}_k(w) = \dd^{T}_k(w) + \binom{B}{k-1} - \binom{n-1-B}{k-1} \textup{ and }
\dd^{T'}_k(v) = \dd^{T}_k(v) - \binom{n-2-B}{k-2}. 
\end{align*}
This gives $\frac{\dd^{T'}_k(w)}{\dd_k^{T'}(v)} = \frac{\dd^{T}_k(w) + \binom{B}{k-1} - \binom{n-1-B}{k-1}}{\dd^{T}_k(v) - \binom{n-2-B}{k-2}}$. As we assumed $vw \notin E(T)$, we know $w_1 \notin V(M_k(T))$. So, $B \leq \frac{n-3}{2}$. When $n-b-1 > k$, the following inequalities each imply the next going from top to bottom.
\begin{align*}
\quad 1 & \geq \frac{\prod_{i=0}^{k-2}(B-i)}{(n-B-k-1) \prod_{i=0}^{k-3}(n-B-2-i)}\\
  n-B-k-1 & \geq \frac{\prod_{i=0}^{k-2}(B-i)}{\prod_{i=0}^{k-3}(n-B-2-i)}\\
% n-B-1 & \geq k-1 + 1 + \frac{\prod_{i=0}^{k-2}(B-i)}{\prod_{i=0}^{k-3}(n-B-2-i)}.\\
\quad\frac{n-B-1}{k-1} & > 1 + \frac{n-k}{(n-1)(k-1)} + \frac{\prod_{i=0}^{k-2}(B-i)}{(k-1)\prod_{i=0}^{k-3}(n-B-2-i)}.
\end{align*}
Converting into binomials gives,
\begin{align*}
\quad\frac{n-B-1}{k-1} & > 1 + \frac{\binom{n-2}{k-1}}{\binom{n-1}{k-1}(k-1)} + \frac{ \binom{B}{k-1}}{\binom{n-B-2}{k-2}}\\
   \frac{\binom{n-B-1}{k-1} - \binom{B}{k-1}}{ \binom{n-B-2}{k-2}} & > 1 + \frac{ \binom{n-2}{k-1}}{ \binom{n-1}{k-1}(k-1)}. 
\end{align*}

Since $\frac{\dd^T_k(w)}{\dd^T_k(v)} \leq 1 + \frac{\binom{n-2}{k-1}}{\binom{n-1}{k-1}(k-1)}$, 
\begin{equation*}
\frac{\dd^{T'}_k(w)}{\dd^{T'}_k(v)} = \frac{\dd^{T}_k(w) + \binom{B}{k-1} - \binom{n-1-B}{k-1}}{\dd^{T}_k(v) - \binom{n-2-B}{k-2}} < \frac{\dd^{T}_k(w)}{\dd^{T}_k(v)},
\end{equation*}
giving a contradiction. 

When $k = n-B-1$, note that $\dd^{T'}_k(w) = \dd^{T}_k(w) - (n-B-1)$ and $\dd^{T'}_k(v)= \dd^{T}_k(v) - (n-B-2)$. Now, note that $n-B-1 > n-B-2 + \frac{B+1}{n-1}$ implies  $\frac{n-B-1}{n-B-2} > 1 + \frac{B+1}{(n-1)(n-B-2)}$. So,  $ \frac{n-B-1}{n-B-2} > 1 + \frac{\binom{n-2}{n-B-2}}{\binom{n-1}{n-B-2}(n-B-2)}.$ However, this implies 
\begin{equation*}
\frac{\dd^{T'}_k(w)}{\dd^{T'}_k(v)} = \frac{\dd^{T}_k(w) - \binom{n-1-B}{n-2-B}}{\dd^{T}_k(v)- \binom{n-2-B}{n-3-B}} < \frac{\dd^{T}_k(w)}{\dd^{T}_k(v)},
\end{equation*}
giving a contradiction. When $k \geq n-B$, we have $\dd_k(w_1) = \dd_k(v)$ which is a contradiction as $w_1$ has less distance to $w$ than $v$. 
So, across all cases, we have $vw \in E(T)$. Thus, we know that $
\frac{\dd_k(w)}{\dd_k(v)} = \frac{\dd_k(v) + \binom{n-2}{k-1}}{\dd_k(v)} = 1 + \frac{ \binom{n-2}{k-1}}{\dd_k(v)}.$

So, we must maximize $\dd_k(v)$ to minimize $\frac{\dd_k(w)}{\dd_k(v)}$. Using Lemma \ref{Theorem: Branches are paths to maximize Steiner distance}, we have that every branch of $v$ is a path. Now, let $T$ be the graph that minimizes $\frac{\dd_k^{T}(w)}{\dd_k^{T}(v)}$ with least number of branches of $v$. Assume for contradiction that in $T$, $v$ has at least $3$ branches that do not include $w$. Let $Q_1$, $Q_2$, and $Q_3$ denote three of these branches and assume without loss of generality that $|V(Q_1)| \leq |V(Q_2)| \leq |V(Q_3)|$. If $|V(Q_3)| > 1 + \frac{n}{2}$, by repeated leaf-transfers from $Q_1$ to $Q_2$ as defined in Lemma \ref{Prep: Moving end vertices of paths} until $Q_1$ disappears,  we can create a graph $S'$. By Theorem \ref{Theorem: Steiner centroid prerequisites}, $S'$ has $v \in V(M_k(S'))$ and $\dd_k^{S'}(v) \geq \dd_k^T(v)$. This is a contradiction. If $|V(Q_3)| \leq 1 + \frac{n}{2}$, there exist leaf transfers from $Q_1$ to $Q_2$ and $Q_1$ to $Q_3$ such that $Q_1$ disappears and $|V(Q_2)| \leq 1 + \frac{n}{2}$ and $|V(Q_3)| \leq 1 + \frac{n}{2}$. Note that by Theorem \ref{Theorem: Steiner centroid prerequisites}, the graph $S''$ created by these leaf-transfers has $v  \in V(M_k(S''))$ and $\dd^{S''}_k(v) \geq \dd^T_k(v)$. Thus, we also have contradiction in this case. So, by contradiction, $v$ has at most $2$ branches that do not include $w$.

There is a unique graph upon which no leaf-transfers that preserve our established properties can be done. That graph must minimize $\frac{\dd_k(w)}{\dd_k(v)}$. By using Theorem \ref{Theorem: Steiner centroid prerequisites} to determine when a leaf-transfer will break the property that $v$ is a Steiner $k$-centroid, we find the following graphs that minimize $\frac{\dd_k(w)}{\dd_k(v)}$ for varying values of $k$. 

When $k \geq n-1$, the graph minimizing $\frac{\dd_k(w)}{\dd_k(v)}$ is a path of length $n-1$ with $w$ as a leaf and $v$ as a vertex adjacent to $w$ as shown in Figure \ref{Fig:x}. This gives $\dd_n(v) = n-1$ and $\dd_{n-1}(v) = 2(n-2) + (n-3)(n-1)$. 

\begin{figure}[htb]
\begin{center}
\begin{tikzpicture}[scale = 0.7]
\draw (0, 1) -- (2,1)-- (4,1) ;
\draw (4, 1) -- (6,1) ;
 \draw [dashed](6,1) -- (8,1) ;
\draw (8, 1) -- (10,1) ;

\fill (0, 1) circle (2pt) node[] {};
\fill (2, 1) circle (2pt) node[] {};
\fill (4, 1) circle (2pt) node[] {};
\fill (6, 1) circle (2pt) node[] {};
\fill (8, 1) circle (2pt) node[] {};
\fill (10, 1) circle (2pt) node[] {};

 \node [below] at (0, 1) {$w$};
  \node [below] at (2, 1) {$v$};

\end{tikzpicture}
\end{center}
 \caption{A graph minimizing $\frac{\dd_k(w)}{\dd_k(v)}$ when $k \geq n-1$.}
 \label{Fig:x}
\end{figure}
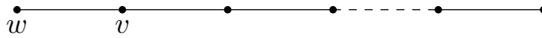

When $n$ is even and $k \leq \frac{n}{2} + 1$, the graph maximizing $\dd_k(v)$ consists of a path of length $n-2$ where $v$ is a vertex on the path of distance $\frac{n}{2}$ from a leaf of the path and a vertex $w$ vertex not on the path that is adjacent to $v$. By applying Lemma \ref{Theorem: a+b=1 path with vertex Steiner distance}, we get $\dd_k(v) = F(\frac{n}{2}, \frac{n}{2}-2)$. Otherwise, when $n$ is odd and $k \leq \frac{n}{2} + 1$, the graph maximizing $\dd_k(v)$ consists of a path of length $n-2$ where $v$ is a vertex on the path of distance $\frac{n-1}{2}$ from a leaf on the path and a vertex $w$ not on the path that is adjacent to $v$. By applying Lemma \ref{Theorem: a+b=1 path with vertex Steiner distance}, we get $\dd_k(v) = F(\frac{n-1}{2}, \frac{n-3}{2})$. When $k > \frac{n}{2} + 1$, the graph maximizing $\dd_k(v)$ consists of a path of length $n-2$ where $v$ is a vertex on the path of distance $k-1$ from a leaf on the path and a vertex $w$ not on the path that is adjacent to $v$. By applying Lemma \ref{Theorem: a+b=1 path with vertex Steiner distance} and simplifying we get $\dd_k(v) = F(k-1, n-k-1)$. This completes the proof.
\end{proof}

\subsection{Extremal ratios between the global and local functions}
\begin{theorem}\label{lemmaMaxRatio}
For a tree $T$ and a centroid $v \in V(T)$, we have $\displaystyle{ \frac{n}{k} \leq \frac{\mathsf{SW}_k(T)}{\dd_k(v)} \leq \frac{n-1}{k-1}}$. 
\end{theorem}

\begin{proof}
First, consider a star $S$ centered at $v$ with $n$ vertices. Then, note that separately considering whether or not the set of $k$ vertices contains $v$ gives $\displaystyle \mathsf{SW}_k(S) = k\binom{n-1}{k} +(k-1)\binom{n-1}{k-1}$. Thus

\begin{equation*}
\frac{\mathsf{SW}_k(S)}{\dd_k(v)} = \frac{k\binom{n-1}{k} +(k-1)\binom{n-1}{k-1}}{(k-1)\binom{n-1}{k-1}} = \frac{n-k}{k-1} + 1 = \frac{n-1}{k-1}.
\end{equation*}
Now, consider a tree T which maximizes $\frac{\mathsf{SW}_k(T)}{\dd_k(v)}$. As the ratio $\frac{\mathsf{SW}_k(T)}{\dd_k(v)} = 1$ for all trees when $k=n$, we will now consider when $k \leq n-1$. Assume for contradiction that $T$ is not a star, and let $u$ be a neighbor of $v$ with more than 1 vertex in its neighborhood. Now, let $u'$ be a neighbor of $u$. We let $B_1$ be the number of vertices in the connected component containing $u$ in $T-uu'-uv$. Also, let $B_2$ be the number of vertices in the connected component containing $u'$ in $T-uu'$. Then, we can represent $T$ as in Figure \ref{fig:minratio}.
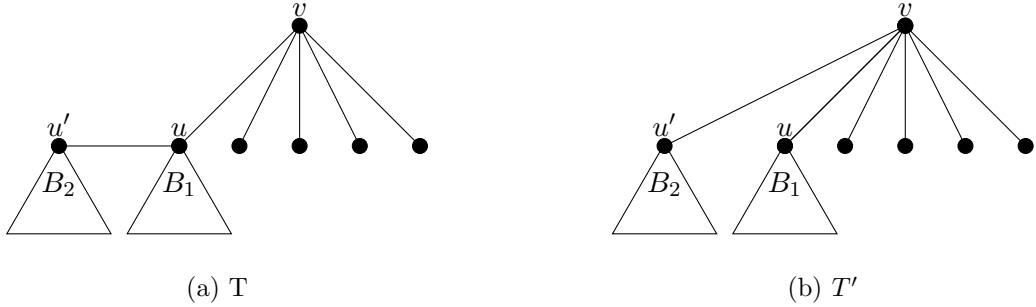
\begin{figure}[htp]
    \centering
    \begin{subfigure}[t]{0.48\textwidth}
         \centering
          \begin{center}
\begin{tikzpicture}[scale = 0.8]
\draw (1,-3) -- (3,-1) 
 node [draw=black, circle, fill=black, inner sep=2pt, pos=0]  {};
\draw (2,-3) -- (3,-1) 
node [draw=black, circle, fill=black, inner sep=2pt, pos=0]  {};
\draw (3,-3) -- (3,-1) 
node [draw=black, circle, fill=black, inner sep=2pt, pos=0]  {};
\draw (4,-3) -- (3,-1)
 node [draw=black, circle, fill=black, inner sep=2pt, pos=0]  {};
\draw  (5, -3) -- (3,-1)
 node [draw=black, circle, fill=black, inner sep=2pt, pos=1]  {}
  node [draw=black, circle, fill=black, inner sep=2pt, pos=0]  {};
\draw[color = black](1,-3) -- (-1,-3)
 node [draw=black, circle, fill=black, inner sep=2pt, pos=1]  {}
 node [draw=black, circle, fill=black, inner sep=2pt, pos=0]  {};
% \draw (1,-5) -- (0,-6.5) -- (2,-6.5) -- (1,-5);
 \node[isosceles triangle, isosceles triangle apex angle=60,
    draw,
    rotate=90,
    minimum size =1.2cm] (T60)at (1,-4){};
  \node [above] at (1,-3) {$u$};
   \node [above] at (-1,-3) {$u'$};
     \node [above] at (3,-1) {$v$};
    \node [above] at (1,-4) {$B_1$};
  \node[isosceles triangle, isosceles triangle apex angle=60,
    draw,
    rotate=90,
    minimum size =1.2cm] (T60)at (-1,-4){};
  \node [above] at (-1,-4) {$B_2$};
  
\end{tikzpicture}
\end{center}
         \caption{T}
         \label{fig:minratioT}
     \end{subfigure}
    \begin{subfigure}[t]{0.48\textwidth}
         \centering
       \begin{center}
\begin{tikzpicture}[scale = 0.8]
\draw (1,-3) -- (3,-1) 
 node [draw=black, circle, fill=black, inner sep=2pt, pos=0]  {};
\draw (2,-3) -- (3,-1) 
node [draw=black, circle, fill=black, inner sep=2pt, pos=0]  {};
\draw (3,-3) -- (3,-1) 
node [draw=black, circle, fill=black, inner sep=2pt, pos=0]  {};
\draw (4,-3) -- (3,-1)
 node [draw=black, circle, fill=black, inner sep=2pt, pos=0]  {};
\draw  (5, -3) -- (3,-1)
 node [draw=black, circle, fill=black, inner sep=2pt, pos=1]  {}
  node [draw=black, circle, fill=black, inner sep=2pt, pos=0]  {};

\draw (1,-3) -- (3,-1)
 node [draw=black, circle, fill=black, inner sep=2pt, pos=1]  {}
 node [draw=black, circle, fill=black, inner sep=2pt, pos=0]  {};
 \draw (-1,-3) -- (3,-1)
 node [draw=black, circle, fill=black, inner sep=2pt, pos=1]  {}
 node [draw=black, circle, fill=black, inner sep=2pt, pos=0]  {};
% \draw (1,-5) -- (0,-6.5) -- (2,-6.5) -- (1,-5);
 \node[isosceles triangle, isosceles triangle apex angle=60,
    draw,
    rotate=90,
    minimum size =1.2cm] (T60)at (1,-4){};
  \node [above] at (1,-4) {$B_1$};
    \node [above] at (1,-3) {$u$};
   \node [above] at (-1,-3) {$u'$};
     \node [above] at (3,-1) {$v$};
  \node[isosceles triangle, isosceles triangle apex angle=60,
    draw,
    rotate=90,
    minimum size =1.2cm] (T60)at (-1,-4){};
  \node [above] at (-1,-4) {$B_2$};
\end{tikzpicture}
\end{center}
         \caption{$T'$}
         \label{fig:minratioT'}
     \end{subfigure}
     
     \caption{A representation of a tree $T$ and $T'$}
    \label{fig:minratio}
\end{figure}

Now consider another tree $T' = T - uv' + vu'$ and note that by Theorem \ref{Theorem: Steiner centroid prerequisites}, $v$ is still a Steiner $k$-centroid in $T'$. We have 
\begin{align*}
\mathsf{SW}_k(T') &= \mathsf{SW}_k(T) - \left( \binom{n-1  }{k} - \binom{B_1}{k} -\binom{n - 1 - B_1}{k}\right) + \left(\binom{B_1 + B_2}{k-1} -\binom{B_1}{k} -\binom{B_2}{k}\right)\\
&= \mathsf{SW}_k(T) - \binom{n-1}{k} +\binom{n-1-B_1}{k} +\binom{B_1 + B_2}{k} -\binom{B_2}{k}.
\end{align*}
Also, we have
\begin{equation*}
    \dd_k^{T'}(v) = \dd_k^T(v) - \binom{n-2}{k-1} + \binom{n-2+B_1}{k-1}.
\end{equation*}
Now, note that
\begin{align*}
     \frac{\binom{n-1}{k} - \binom{n-1-B_1}{k} - \binom{B_1 + B_2}{k} + \binom{B_2}{k}}{\binom{n-2}{k-1} - \binom{n-2-B_1}{k-1}}
    & < \frac{(n-1)\cdots(n-k) - (n-1-B_1)\cdots(n-k-B_1)}{k \left( (n-2)\cdots(n-k) - (n-2-B_1)\cdots(n-k-B_1)\right)} \\
   & < \frac{(n-1)\cdots(n-k) - (n-1)(n-2-B_1)\cdots(n-k-B_1)}{k \left( (n-2)\cdots(n-k) - (n-2-B_1)\cdots(n-k-B_1)\right)} \\
    & = \frac{n-1}{k}\\
    & < \frac{n-1}{k-1}.
\end{align*}
Since $\frac{\mathsf{SW}_k(T)}{\dd^T_k(v)} \geq \frac{n-1}{k-1}$,
\begin{equation*}
\frac{\mathsf{SW}_k(T')}{\dd_k^{T'}(v)} > \frac{\mathsf{SW}_k(T)}{\dd^T_k(v)}, 
\end{equation*}
which is a contradiction.
Thus, we know that $T$ is a star and by our previous calculations, $\frac{\mathsf{SW}_k(T)}{\mathsf{d}^T_k(v)} = \frac{n-1}{k-1}.$ 

Now, to establish a lower bound, consider an arbitrary tree $T$ and denote its vertices as $v_1, v_2, \dots, v_n$. We note that $\displaystyle{{\mathsf{SW}_k(T) = \frac{1}{k}\sum_{i=1}^{n}\mathsf{d}_k(v_i)}}$. This implies that $\frac{\mathsf{SW}_k(T)}{\mathsf{d}^T_k(v)} \geq \frac{n}{k}$. 
\end{proof}

\section{Concluding Remarks}
We studied a number of questions related to the Steiner $k$-distance and the corresponding middle parts and extremal ratios in trees. These results generalize previous work on regular distances. 

Some natural questions arise from our study of the extremal bounds, where we wish to maximize the function $F$ in Theorems~\ref{Theorem: Bounds for ratio of leafs} and ~\ref{Theorem: Bounds for internal vertices}.

Bounding $\frac{\mathsf{SW}_k(T)}{\dd_k(w)}$ where $w \in V(M_k(T))$ is another logical next step for this paper. With such bounds and the corresponding extremal structures, comparisons can be made with other topological indices.

\end{document}